\documentclass[11pt, letterpaper]{amsart}
\usepackage[left=1in,right=1in,bottom=0.5in,top=0.8in]{geometry}
\usepackage{amsfonts}
\usepackage{amsmath, amssymb}
\usepackage{graphicx}
\usepackage[font=small,labelfont=bf]{caption}
\usepackage{epstopdf}
\usepackage[pdfpagelabels,hyperindex]{hyperref}
\usepackage{xcolor}
\usepackage{amsthm}
\usepackage{float}
\usepackage{pgfplots}
\usepackage{listings}
\usepackage{longtable}
\usepackage{mathrsfs}

\hypersetup{
pdftitle={Another Power Identity involving Binomial Theorem and Faulhaber's formula},
pdfsubject={Mathematics, Number Theory, Combinatorics},
pdfauthor={Petro Kolosov},
pdfkeywords={Faulhaber's formula, Faulhaber's theorem, Binomial Theorem, Binomial coefficient, Binomial distribution,
Binomial identities, Power Sums, Finite differences}
}

\newtheorem{thm}{Theorem}[section]

\newtheorem{lem}[thm]{Lemma}

\theoremstyle{definition}

\theoremstyle{remark}
\newtheorem{rem}[thm]{Remark}

\definecolor{energy}{RGB}{114,0,172}
\definecolor{freq}{RGB}{45,177,93}
\definecolor{spin}{RGB}{251,0,29}
\definecolor{signal}{RGB}{203,23,206}
\definecolor{circle}{RGB}{217,86,16}
\definecolor{average}{RGB}{203,23,206}

\colorlet{shadecolor}{gray!20}
\pgfplotsset{compat=1.9}

\usepgflibrary{fpu}

\newcommand{\RN}[1]{%
  \textup{\uppercase\expandafter{\romannumeral#1}}%
}
\newcounter{x}
\def\ToRomanEmpire#1{\setcounter{x}{#1}\Roman{x}}


\numberwithin{equation}{section}

\author[Junfu Yao]{Junfu Yao}
\email{jyao21@jhu.edu}

\title{A Uniqueness Result for Self-expanders with Small Entropy}
\begin{document}
\begin{abstract}
In this short note, we prove a uniqueness result for small entropy self-expanders asymptotic to a fixed cone. This is a direct consequence of the mountain-pass theorem and the integer degree argument proved by J. Bernstein and L. Wang.

\end{abstract}
\maketitle

\section{Introduction}
\indent{A properly embedded $n$-dimensional submanifold $\Sigma$ in $\mathbb{R}^{n+1}$ is called a $\emph{self-expander}$ if it satisfies}
\begin{equation}
    \textbf{H}_\Sigma=\frac{\textbf{x}^{\perp}}{2}\label{1.1}
\end{equation}
where $\textbf{H}_\Sigma$ is the mean curvature vector of $\Sigma$, and $\textbf{x}^\perp$ is the normal component of the position vector. Self-expanders are self similar solutions of $\emph{mean curvature flow}$, that is, the family of hypersurfaces
\begin{equation*}
    \{\Sigma_t\}_{t>0}=\big\{\sqrt{t}\Sigma\big\}_{t>0}
\end{equation*}
satisfying
\begin{equation*}
    \Big(\frac{\partial \textbf{x}}{\partial t}\Big)^{\perp}=\textbf{H}_{\Sigma_t}
\end{equation*}
Self-expanders are important as they model the behavior of a mean curvature flow coming out of a conical singularity \cite{1}, and also model the long time behaviors of the flows starting from entire graphs \cite{8}.
\par\indent{For a hypersurface $\Sigma$ in $\mathbb{R}^{n+1}$, Colding-Minicozzi \cite{12} introduced the entropy on $\Sigma$}
\begin{equation}
    \lambda[\Sigma]=\sup\limits_{\textbf{y}\in\mathbb{R}^{n+1}, \rho>0}F[\rho\Sigma+\textbf{y}]
\end{equation}
where $F[\Sigma]$ is the $\emph{Gaussian surface area}$ of $\Sigma$
\begin{equation*}
    F[\Sigma]=(4\pi)^{-\frac{n}{2}}\int_{\Sigma}e^{-\frac{|\textbf{x}|^2}{4}}d\mathcal{H}^n(\textbf{x})
\end{equation*}
Obviously, this quantity is invariant under dilations and translations. And Huisken's monotonicity formula \cite{9} shows that this quantity is non-increasing along the mean curvature flow.
\par\indent{Next, we talk about the space $\mathcal{ACH}^{k,\alpha}_n$ introduced by J. Bernstein and L. Wang \cite{2}. A hypersurface $\Sigma\in \mathcal{ACH}^{k,\alpha}_n$, if it is a $C^{k,\alpha}$ properly embedded codimension-one submanifold and $C^{k,\alpha}_*$-asymptotic to a $C^{k,\alpha}$ regular cone $\mathcal{C}=\mathcal{C}(\Sigma)$. We refer to \cite[Section 2]{2} for technical details.}
\par{As pointed out in \cite{3} and \cite{5}, there may exist more than one self-expanders asymptotic to some specific cones. While the main theorem here states that for a small entropy cone, there's only one (stable) self-expander asymptotic to it. More precisely,}
\begin{thm}
There exists a constant $\delta=\delta(n)$, so that for a given $C^{k,\alpha}$-regular cone $\mathcal{C}$ with $\lambda[\mathcal{C}]<1+\delta$, there is a unique stable self-expander $\Sigma \in \mathcal{ACH}^{k,\alpha}_n$ with $\mathcal{C}(\Sigma)=\mathcal{C}$.\label{main}
\end{thm}
\begin{rem}
As pointed out in \cite[Theorem 8.21]{14}, the outermost flows of any hypercone are self-expanders. Hence, it follows from our result that for a low entropy cone, the inner and outer flows coincide, so the cone doesn't fatten.
\end{rem}

\section{Some regularity results}
\par\indent{In \cite{6}, the authors defined a space $\mathcal{RMC}_n$, consisting of all regular minimal cones in $\mathbb{R}^{n+1}$ and let $\mathcal{RMC}^{*}_n$ be the non-flat elements in $\mathcal{RMC}_n$. For any $\Lambda>1$, let}
\begin{equation*}
    \mathcal{RMC}_n(\Lambda)=\{\mathcal{C}\in \mathcal{RMC}_n:\ \lambda[\mathcal{C}]<\Lambda\}\ and\ \mathcal{RMC}^{*}_n(\Lambda)=\mathcal{RMC}^{*}_n\cap\mathcal{RMC}_n(\Lambda)
\end{equation*}
Since all regular minimal cones in $\mathbb{R}^{2}$ consist of the unions of rays and great circles are the only geodesics in $\mathbb{S}^2$, $\mathcal{RMC}^*_1(\Lambda)=\mathcal{RMC}^*_2(\Lambda)=\emptyset$ for all $\Lambda>1$. Now fix the dimension $n\geq 3$ and a value $\Lambda>1$. Consider the following hypothesis:
\begin{equation}
    \text{For all }3\leq l\leq n,\ \mathcal{RMC}^{*}_l(\Lambda)=\emptyset \label{2.1}
\end{equation}
\begin{lem}
There is a constant $\Lambda=\Lambda_n>1$, so that the hypothesis \textnormal{($\ref{2.1}$)} holds.
\end{lem}
\begin{proof}
We first show that any regular minimal cone $\mathcal{C}$ has generalized mean curvature 0 near the origin. Let $X$ be a smooth vector field compactly supported in $\mathbb{R}^{n+1}$. Define a smooth cut-off function $\eta$, so that $\eta=0$ in $B_{\frac{1}{2}}(0)$ and $\eta=1$ outside $B_1(0)$. Let $\eta_r=\eta(\frac{\cdot}{r})$, then
\begin{equation}
\begin{split}
        \int_{\mathcal{C}}\eta_r\big(div_{\mathcal{C}}X\big)d\mathcal{H}^n&=\int_{\mathcal{C}}div_{\mathcal{C}}\big(\eta_r X\big)d\mathcal{H}^n-\int_{\mathcal{C}}X\cdot \nabla_{\mathcal{C}}\eta_r d\mathcal{H}^n \label{2.2}\\
        &=-\int_{\mathcal{C}}X\cdot \nabla_{\mathcal{C}}\eta_r d\mathcal{H}^n
\end{split}
\end{equation}
The construction of $\eta_r$ gives 
\begin{equation*}
    Spt\ \nabla \eta_r\subset \overline{B_{r}(0)}\backslash B_{\frac{r}{2}}(0)\ and\ |\nabla_\mathcal{C}\eta_r |\leq |\nabla \eta_r| \leq \frac{C}{r}
\end{equation*}
where $C=C(n)$ is a constant. Let $\mathcal{L}(\mathcal{C})=\mathcal{C}\cap \mathbb{S}^n$ be the regular codimension-1 submanifold in $\mathbb{S}^n$, then $\mathcal{H}^{n-1}(\mathcal{L}(\mathcal{C}))<\infty$. This gives a upper bound for the last term in ($\ref{2.2}$)
\begin{equation*}
\begin{split}
    |\int_{\mathcal{C}}X\cdot \nabla_{\mathcal{C}}\eta_r d\mathcal{H}^n|&\leq \parallel X\parallel_{\infty} |\int_{\mathcal{C}\cap B_{r}(0)\backslash B_{\frac{r}{2}}(0)}\frac{C}{r}|d\mathcal{H}^n\\
    &\leq \parallel X\parallel_{\infty}\cdot \frac{C}{r}\mathcal{H}^{n-1}(\mathcal{L}(\mathcal{C}))\int_{\frac{r}{2}}^r s^{n-1}ds\\
    &\leq C\mathcal{H}^{n-1}\mathcal{L}(\mathcal{C})\parallel X\parallel_{\infty}r^{n-1}
\end{split}
\end{equation*}
So it goes to 0 as $r\rightarrow 0$. Letting $r\rightarrow 0$ in ($\ref{2.2}$), we get
\begin{equation}
    \int_{\mathcal{C}}div_{\mathcal{C}}X d\mathcal{H}^n=0
\end{equation}
for any vector field compactly supported in $\mathbb{R}^{n+1}$, which means $\mathcal{C}$ has generalized mean curvature near the origin and it vanishes.
\par\indent{Next, we relate the entropy to the density at origin. Observe that}
\begin{equation*}
    \begin{split}
        \frac{1}{(4\pi)^{\frac{n}{2}}}\int_{\mathcal{C}}e^{-\frac{|\textbf{x}|^2}{4}}d\mathcal{H}^{n}&=\frac{1}{(4\pi)^{\frac{n}{2}}}\mathcal{H}^{n-1}(\mathcal{L}(\mathcal{C}))\int_0^\infty r^{n-1}e^{-\frac{r^2}{4}}dr\\
        &=\frac{1}{2\pi^{\frac{n}{2}}}\mathcal{H}^{n-1}(\mathcal{L}(\mathcal{C}))\Gamma(\frac{n}{2})\\
        &=\frac{\mathcal{H}^{n-1}(\mathcal{L}(\mathcal{C}))}{n\omega_n}\\
        &=\frac{\mathcal{H}^{n}(B_{\rho}(0))}{\omega_n \rho^n}\\
        &=\Theta(\mathcal{C},0)
    \end{split}
\end{equation*}
where $\omega_n=\frac{\pi^{\frac{n}{2}}}{\Gamma(\frac{n}{2}+1)}$ is the volume of the unit ball in $\mathbb{R}^{n}$.
So the density $\Theta(\mathcal{C},0)=F[\mathcal{C}]\leq \lambda[\mathcal{C}]$. Thus, by Allard's regularity theorem \cite{10}, if $\Lambda_n$ is sufficiently small, then $\mathcal{C}$ is smooth at the origin, and it has to be a hyperplane.
\end{proof}
\begin{rem}
In fact, if we replace the regular cone above by a general stationary integral varifold $\mathcal{C}$ with $\eta_{0,\rho\#}\mathcal{C}=\mathcal{C}$ for all $\rho>0$, the result still holds. Indeed, we get the smoothness near the origin in the same way and the dilation invariance implies it is smooth everywhere.
\end{rem}
\par\indent{The following is a lemma from \cite{6}. For the sake of completeness, we include a proof here.}
\begin{lem}
For $k\geq 2$ and $\alpha\in (0,1)$, let $\mathcal{C}$ be a $C^{k,\alpha}$-regular cone in $\mathbb{R}^{n+1}$ and assume $\lambda[\mathcal{C}]<\Lambda_n$. If $V$ is an $E$-stationary integral varifold with tangent cone at infinity equal to $\mathcal{C}$, then $V=V_\Sigma$ for an element $\Sigma\in \mathcal{ACH}^{k,\alpha}_n$ satisfying (\ref{1.1}).\label{regul}
\end{lem}
\begin{proof}
For every point $x\in Spt \mu_V$, we will show that there is a tangent plane $\mathcal{P}_x$. If that is the case, $\Theta^{n}(V,x)=\Theta^n(\mathcal{P}_x,0)=1$. Together with the fact that $V$ has locally bounded mean curvature, the Allard's regularity theorem applies and $V$ is smooth near $x$.
\par{Suppose a sequence of positive number $\{\lambda_i\}\rightarrow 0$, and $\eta_{x,\lambda_i\#} V$ converges to a tangent varifold $\mathcal{C}_x$, where $\eta_{x,\lambda_i}(y)=\frac{y-x}{\lambda_i}$. By the nature of convergence and Huisken's monotonicity formula \cite{9},}
\begin{equation*}
    \lambda[\mathcal{C}_x]\leq \lambda[V]\leq \lambda[\mathcal{C}]<\Lambda_n
\end{equation*}
On the other hand,
\begin{equation*}
\begin{split}
    \rho^{1-n}\parallel\delta V\parallel (B_{\rho}(x))&\leq \rho^{1-n}\int_{V\cap B_{\rho}(x)}|\textbf{H}_V|d\mathcal{H}^n\\
    &\leq \parallel \textbf{H}_V\parallel_{\infty}\rho\cdot\frac{\mathcal{H}^n(V\cap B_{\rho}(x))}{\rho^n}
\end{split}
\end{equation*}
It goes to 0 as $\rho\rightarrow 0^+$. Then it follows from the knowledge of varifold in \cite{10} that $\eta_{0,\rho\#}\mathcal{C}_x=\mathcal{C}_x$ for $\rho>0$, and that $\mathcal{C}_x$ is stationary. The previous lemma indicates $\mathcal{C}_x$ must be a plane. Thus $V=V_\Sigma$ for some smooth self-expander $\Sigma$. Following \cite{4}, $\Sigma$ is $C^{k,\alpha}_*$-asymptotic to $\mathcal{C}$.
\end{proof}
\par\indent{We also need the notion of partial ordering. Roughly speaking, $\Sigma_1\preceq\Sigma_2$, if the hypersurface $\Sigma_1$ is ``above" $\Sigma_2$. For the detailed explanation, we refer to \cite[Section 4]{6}. The following theorem from that paper is a useful tool to construct self-expanders:}
\begin{thm}
For $k\geq 2$ and $\alpha \in (0,1)$, let $\mathcal{C}$ be a $C^{k,\alpha}$-regular cone in $\mathbb{R}^{n+1}$ and assume either $2\leq n\leq 6$ or $\lambda[\mathcal{C}]<\Lambda_n$. For any two $\Sigma_1$ and $\Sigma_2 \in \mathcal{ACH}^{k,\alpha}_n$, with $\mathcal{C}(\Sigma_1)=\mathcal{C}(\Sigma_2)=\mathcal{C}$, there exist $\Sigma_{\pm}$ stable self-expanders asymptotic to $\mathcal{C}$ with $\Sigma_{-}\preceq \Sigma_{i}\preceq \Sigma_{+}$ for $i=1,2$.\label{constr}
\end{thm}

\section{The Relationship Between Entropy and Stability}
\par\indent{As oberved by S. Guo \cite{13}, if the entropy of the cone is sufficiently small, then we have the curvature bound for the self-expander. More precisely,}
\begin{thm}
Given $\kappa>0$, there exists a constant $\epsilon>0$ depending on $n$ and $\kappa$ with the following property.
\par\indent{If $\Sigma$ is a self-expander that is asymptotic to a regular cone $\mathcal{C}$ with $\lambda[\mathcal{C}]<1+\epsilon$. Then we have $\parallel A_{\Sigma}\parallel_{L^{\infty}}\leq \kappa$.}\label{Guo}
\end{thm}
Guo pointed out that when $\kappa<\frac{1}{\sqrt{2}}$, all self-expanders are stable.  In the following lemma, we use a $\textnormal{Poincar\'{e}}$ type inequality from \cite{3} to get a slightly strengthening by improving the bound to $\kappa\leq\sqrt{\frac{n+1}{2}}$ .
\begin{lem}
Let $\Sigma$ be a self-expander in $\mathcal{ACH}^{k,\alpha}_n$ satisfying $|A_{\Sigma}|^2\leq \frac{n+1}{2}$. Then $\Sigma$ is strictly stable in the sense that for all $u\in C^{\infty}_{c}(\Sigma)\backslash\{0\}$,
\begin{equation}
    \langle -\mathcal{L}_{\Sigma}u, u\rangle =\int_{\Sigma}\Big[|\nabla_{\Sigma}u|^2+(\frac{1}{2}-|A_{\Sigma}|^2)u^2\Big]e^{\frac{1}{4}|\textbf{x}|^2}d\mathcal{H}^n> 0
\end{equation}\label{Poincare}
\end{lem}
\begin{proof}
Following \cite[Appendix A]{3}, since $(\Delta_{\Sigma}+\frac{1}{2}\textbf{x}\cdot\nabla_{\Sigma})(|\textbf{x}|^2+2n)=|\textbf{x}|^2+2n$, integrating by parts gives
\begin{equation}
    \begin{split}
        &\int_{\Sigma}(2n+|\textbf{x}|^2)u^2e^{\frac{1}{4}|\textbf{x}|^2}d\mathcal{H}^n=\int_{\Sigma}\Big[(\Delta_{\Sigma}+\frac{1}{2}\textbf{x}\cdot\nabla_{\Sigma})(|\textbf{x}|^2+2n)\Big]u^2e^{\frac{1}{4}|\textbf{x}|^2}d\mathcal{H}^2\\
        &=-\int_{\Sigma}\nabla_{\Sigma}(|\textbf{x}|^2+2n)\cdot\nabla_{\Sigma}(u^2)e^{\frac{1}{2}|\textbf{x}|^2}d\mathcal{H}^n=-4\int_{\Sigma}u\ \textbf{x}^{\top}\cdot \nabla_{\Sigma}u\ e^{\frac{1}{4}|\textbf{x}|^2}d\mathcal{H}^n\\
        &\leq \int_{\Sigma}(|\textbf{x}^{\top}|^2u^2+4|\nabla_{\Sigma}u|^2)e^{\frac{1}{4}|\textbf{x}|^2}d\mathcal{H}^n. \label{3.2}
    \end{split}
\end{equation}
So moving $|\textbf{x}^{\top}|^2u^2$ to the left hand side,
\begin{equation*}
    \int_{\Sigma }(2n+|\textbf{x}^{\perp }|^2)u^2e^{\frac{1}{4}|\textbf{x}|^2}d\mathcal{H}^n\leq 4\int_{\Sigma }|\nabla_\Sigma u|^2e^{\frac{1}{4}|\textbf{x}|^2}d\mathcal{H}^n
\end{equation*}
Together with $|A_\Sigma |^2\leq \frac{n+1}{2}$,
\begin{equation*}
\begin{split}
    \langle -\mathcal{L}_\Sigma u,u\rangle &=\int_\Sigma \Big[|\nabla_\Sigma u|^2+(\frac{1}{2}-|A_\Sigma|^2)u^2\Big]e^{\frac{1}{4}|\textbf{x}|^2}d\mathcal{H}^n\\
    &\geq\int_\Sigma (|\nabla_\Sigma u|^2-\frac{n}{2}u^2)e^{\frac{1}{4}|\textbf{x}|^2}d\mathcal{H}^n\\
    &\geq\frac{1}{4}\int_{\Sigma}|\textbf{x}^\perp|^2u^2e^{\frac{1}{4}|\textbf{x}|^2}d\mathcal{H}^n\geq 0
\end{split}
\end{equation*}
If there were $u\in C^\infty_c(\Sigma)\backslash \{0\}$ satisfying $\langle -\mathcal{L}_\Sigma u,u\rangle=0$, then the inequlity (\ref{3.2}) above should be an equality, which means 
\begin{equation}
    u\textbf{x}^{\top}=-2\nabla_\Sigma u,\ \textnormal{in }\{u>0\}\label{equa case}
\end{equation}
Fix $p\in \{u>0\}$, $r=\sup\{s>0:\ B^\Sigma_s(p)\subset\{u>0\}\}$ and define $v=\log{u}$ in $B^\Sigma_r(p)$. $u$ being compactly supported implies $r<\infty$ is well-defined and that there is a $p_0\in \partial B_{r}(p)\cap\{u=0\}$. From (\ref{equa case}) we know that $\nabla_\Sigma (v+\frac{1}{4}|\textbf{x}|^2)=0$, which means $v=-\frac{1}{4}|\textbf{x}|^2+constant $. However, this contradicts to the fact that $v\rightarrow -\infty$ as $q\rightarrow p_0$. Hence, $\langle -\mathcal{L}_\Sigma u,u\rangle>0$ for all non-trivial $u$.
\end{proof}

\section{Proof of Theorem \ref{main}}
\par\indent{In this section we use the mountain-pass theorem proved by J. Bernstein and L. Wang to prove Theorem \ref{main}.}
\begin{proof}[Proof of Theorem \ref{main}]
The existence result follows from \cite[Theorem 6.3]{7} and \cite[Proposition 3.3]{4}. Hence, we only need to show the uniqueness. Letting $\kappa =\sqrt{\frac{n+1}{2}}$, Theorem \ref{Guo} ensures the existence of an $\epsilon=\epsilon(n)$ so that if $\lambda[\mathcal{C}]<1+\epsilon$, then $|A_\Sigma|\leq \kappa$ for all self-expanders asymptotic to $\mathcal{C}$. Then by Lemma \ref{Poincare}, all self-expanders asymptotic to $\mathcal{C}$ are strictly stable. That is, $\mathcal{C}$ is a \emph{regular value} in the sense of \cite{3}.
\par{Let $\delta=\min\{\epsilon(n), \Lambda_n-1\}$. As an application of \cite[Theorem 1.3]{3}, $\Pi^{-1}(\mathcal{C})$ is a finite set, where $\Pi$ assigns each element in $\mathcal{ACH}^{k,\alpha}_n$ to the trace at infinity.}
\par{Now, let us argue by contradiction. Suppose there were two self-expanders $\Sigma_1$ and $\Sigma_2$ both asymptotic to $\mathcal{C}$. Following Theorem \ref{constr}, we can produce two distinct self-expanders $\Sigma_\pm$ with $\Sigma_{-}\preceq\Sigma_i\preceq\Sigma_+$ for $i=1,2$. Applying \cite[Theorem 1.1]{5}, there is a self-expander $\Sigma_0\neq\Sigma_\pm$ with possibly codimension-7 singular set and $\Sigma_-\preceq\Sigma_0\preceq\Sigma_+$. Notice that Huisken's monotonicity formula tells us $\lambda[\Sigma_0]\leq \lambda[\mathcal{C}]<1+\delta\leq \Lambda_n$. Thus, by Lemma \ref{regul}, $\Sigma_0$ is actually smooth. Now, replace $\Sigma_\pm$ by $\Sigma_0$ and $\Sigma_-$ and iterate the preceding argument. So we produce as many self-expanders as we can. And this contradicts the fact that $\Pi^{-1}(\mathcal{C})$ is finite.}
\end{proof}

\section{Acknowledgements}
The author would like to thank Professor Jacob Bernstein for his patient and generous help.


\begin{thebibliography}{9}
\bibitem{1}
        S.B. Angenent, T. Ilmanen, and D.L. Chopp, \textit{A computed example of non-uniqueness of mean curvature flow in $\mathbb{R}^3$}, Commun. in Partial Differential Equations 20 (1995), no. 11-12, 1937-1958.
\bibitem{2}
        J. Bernstein and L. Wang, \textit{The space of asymptotically conical self-expanders of mean curvature flow}, Preprint. Available at \href{https://arxiv.org/abs/1712.04366}{https://arxiv.org/abs/1712.04366}.
\bibitem{3}
        J. Bernstein and L. Wang, \textit{An integer degree for asymptotically conical self-expanders}, Preprint. Available at \href{https://arxiv.org/abs/1807.06494}{https://arxiv.org/abs/1807.06494}.
\bibitem{4}
        J. Bernstein and L. Wang, \textit{Smooth compactness for spaces of asymptotically conical sel-expanders of mean curvature flow}, Int. Math. Res. Not., to appear. Available at \href{https://arxiv.org/abs/1804.09076}{https://arxiv.org/abs/1804.09076}.
\bibitem{5}
        J. Bernstein and L. Wang, \textit{A mountain-pass theorem for asymptotially conical self-expanders}, Preprint. Available at \href{https://arxiv.org/abs/2003.13857}{https://arxiv.org/abs/2003.13857}.
\bibitem{6}
        J. Bernstein and L. Wang, \textit{Topological uniqueness for self-expanders of small entropy}, Preprint. Available at \href{https://arxiv.org/abs/1902.02642}{https://arxiv.org/abs/1902.02642}.
\bibitem{14}
        Otis Chodosh, Kyeongsu Choi, Christos Mantoulidis, and Felix Schulze, \textit{Mean curvature flow with generic initial data}. Available at \href{https://arxiv.org/abs/2003.14344}{https://arxiv.org/abs/2003.14344}.
\bibitem{12}
        T.H. Colding and W.P. Minicozzi $\RN{2}$, \textit{Generic mean curvature flow $\ToRomanEmpire{1}$; generic singularities}, Ann. of Math. (2) 175 (2012), no. 2, 755-833.
\bibitem{7}
        Q. Ding, \textit{Minimal cones and self-expanding solutions for mean curvature flows}, Math. Ann. 376 (2020), no. 1-2, 359-405.
\bibitem{8}
        K. Ecker and G. Huisken, \textit{Mean curvature evolution of entire graphs}, Ann. of Math. (2) 130 (1989), no. 3, 453-471.
\bibitem{13}
        S.-H. Guo, \textit{Asymptotic behavior and stability of mean curvature flow with a conical end}, Adv. Math (2020), \href{https://doi.org/10.1016/j.aim.2020.107408}{https://doi.org/10.1016/j.aim.2020.107408}
\bibitem{9}
        G. Huisken, \textit{Asymptotic behavior for singularities of the mean curvature flow}, J. Differential Geom. 31 (1990), no. 1, 285-299.
\bibitem{10}
        L. Simon, Lectures on geometric measure theory, Proceedings of the Centre for Mathematical Analysis, Australian National University No. 3, Canberra, 1983.
\bibitem{11}
        B. White, \textit{Stratification of minimal surfaces, mean curvature flows, and harmonic maps}, J. Reine Angew. Math. 488 (1997), 1-35.

\end{thebibliography}
\end{document}